\documentclass[12pt]{amsart}

\usepackage{amssymb}

\usepackage{enumitem}

\usepackage{graphicx}
\usepackage{amsfonts}
\usepackage{mathrsfs}
\usepackage{amsmath}
\usepackage{amssymb}
\usepackage{amsthm}
\usepackage{enumitem}
\usepackage{latexsym}
\usepackage{color}
\usepackage[margin=0.8in]{geometry}
\usepackage{hyperref}

\theoremstyle{definition}
\newtheorem{defin}{Definition}[section]
\newtheorem{prop}[defin]{Proposition}
\newtheorem{lem}[defin]{Lemma}

\newtheorem{teo}[defin]{Theorem}

\newtheorem{question}[defin]{Question}

\DeclareMathOperator{\dom}{dom}
\DeclareMathOperator{\supp}{supp}
\DeclareMathOperator{\Ult}{Ult}
\DeclareMathOperator{\cf}{cf}

\newcommand{\bs}{\backslash}

\newcommand{\Q}{\mathbb Q}

\newcommand{\bbT}{\mathbb T}

\newcommand{\calA}{\mathcal A}

\newcommand{\calC}{\mathcal C}

\newcommand{\calE}{\mathcal E}

\newcommand{\calG}{\mathcal G}
\newcommand{\calH}{\mathcal H}

\newcommand{\calP}{\mathcal P}
\newcommand{\calQ}{\mathcal Q}

\newcommand{\calS}{\mathcal S}

\newcommand{\calU}{\mathcal U}

\newcommand{\frakc}{\mathfrak c}

\makeatletter
\@namedef{subjclassname@2010}{%
  \textup{2010} Mathematics Subject Classification}
\makeatother

\usepackage[T1]{fontenc}
\numberwithin{equation}{section}

\begin{document}
	
	\baselineskip=17pt
	\title[$\mathcal U$-compact groups of countable cofinality size]{$\calU$-compact topologies without convergent sequences on countably cofinal torsion-free Abelian groups}
	
	%\title[ group topologies on direct copies of ${\mathbb Q}$]{Group topologies without non-trivial convergent sequences on arbitrarily large direct sums of ${\mathbb Q}$ whose every power is countably compact.}
	\author[M. K. Bellini]{Matheus Koveroff Bellini}
	
	\author[A. H. Tomita]{Artur Hideyuki Tomita}
	\address{Depto de Matem\'atica, Instituto de Matem\'atica e Estat\'istica, Universidade de S\~ao Paulo, Rua do Mat\~ao, 1010 -- CEP 05508-090, S\~ao Paulo, SP - Brazil}
	\email{tomita@ime.usp.br, matheusb@ime.usp.br}
	\thanks{The first listed author has received financial support from FAPESP 2017/15709-6.}
	\thanks{The second listed author has received financial support from FAPESP 2021/00177-4.}
	\date{}
	\commby{}
	\keywords{Topological group, countable compactness, selective ultrafilter, ${\mathbb Q}$-vector spaces, Wallace's problem}

	\begin{abstract}
	We obtain a forcing construction that shows that it is consistent that the torsion-free Abelian group $\Q^{(\lambda)}$ admits a Hausdorff group topology which is also $\calU$-compact and contains no non-trivial convergent sequences, where $\lambda$ is a cardinal whose cofinality is $\omega$ and $\calU$ is a selective ultrafilter. This answers a question posed in \cite{bellini&rodrigues&tomita2021.Qkappa}.
	\end{abstract}
	
	\maketitle

\section{Introduction}

The study of compactness properties in topological groups has been carried on for more than 70 years. In a seminal paper by Halmos, \cite{halmos}, it was shown that the real line can be endowed with a compact group topology. We note that algebraically the real line is the sum of $\frakc$ copies of the rationals, not only as an Abelian group but also as a $\Q$-vector space (a fact that is used in Halmos's result and in much of the subsequent work in the matter).

We recall that, given an ultrafilter $\calU$ and a topological space $X$, a $\calU$-\textit{limit} of a sequence $(x_n)_{n\in\omega}$ in $X$ is a $z\in X$ such that for every neighbourhood $A$ of $z$, $\{n\in\omega:x_n\in A\}\in\calU$. A topological space $X$ is called $\calU$-\textit{compact} if every sequence has a $\calU$-limit.

In the paper \cite{castro-pereira&tomita2010}, Castro and Tomita obtained a consistent classification -- assuming some cardinal arithmetic and the existence of a selective ultrafilter $p$ -- of the Abelian torsion groups which admit a $p$-compact group topology. This proved the consistency of large countably compact torsion groups without non-trivial convergent sequences.

A more recent result, \cite{bellini&rodrigues&tomita2021.Qkappa}, showed that given a selective ultrafilter $p$, $\Q^{(\kappa)}$ can be endowed with a $p$-compact Hausdorff group topology without non-trivial convergent sequences, whenever $\kappa=\kappa^\omega$; this hypothesis implies that $\cf(\kappa)>\omega$. It remained open then whether such a result could consistently hold for cardinals of countable cofinality.

This article answers that question, using forcing to obtain a model where $\lambda$ is a cardinal whose cofinality is $\omega$ and such that $\Q^{(\lambda)}$ has a $\calU$-compact Hausdorff group topology without non-trivial convergent sequences, where $\calU$ is a given selective ultrafilter.

\section{Notation}

We shall fix throughout this article a cardinal $\lambda$ and a selective ultrafilter $\calU$.

As usual, given $a\in\Q^\lambda$, its \textit{support} is $\supp a=\{\xi\in\lambda:a_\xi\ne0\}$.

Let $G$ be the Abelian group $\Q^{(\lambda)}=\{a\in\Q^\lambda:\supp a\text{ is finite}\}$ (considering coordinatewise addition as its operation).

If $E\subseteq\lambda$, we do a standard abuse of notation and consider $\Q^{(E)}=\{a\in G:\supp a\subseteq E\}$.

Given $\mu\in\lambda$, we define $\chi_\mu\in G$ by $\chi_\mu(\mu)=1$ and $\chi_\mu(\beta)=0$ for all $\beta\in\lambda$, $\beta\ne\mu$. Now, given $\zeta:\omega\to\lambda$, we define $\chi_\zeta:\omega\to G$ by $\chi_\zeta(n)=\chi_{\zeta(n)}$ for each $n\in\omega$. And given a $\mu\in\lambda$, we define $\vec\mu$ as the constant sequence of value $\mu$. Thus, $\chi_{\vec\mu}$ is the constant sequence of value $\chi_\mu$ (this will appear often in our work).

Since $\calU$ is an ultrafilter on $\omega$, the \textit{ultrapower} of $G$ by $\calU$, denoted $\Ult_\calU(G)$, is the quotient of the set $G^\omega$ by the following equivalence relation: $g\sim h$ if and only if $\{n\in\omega:g(n)=h(n)\}\in\calU$. We will make frequent use of the fact that $\Ult_\calU(G)$ is a $\Q$-vector space with all operations defined naturally (that is via representatives). If $g\in G^\omega$, we will denote its class in $\Ult_\calU(G)$ by $[g]_\calU$.

We now fix $\calH\subseteq G^\omega$ such that $([g]_\calU:g\in\calH)\cup([\chi_{\vec\mu}]_\calU:\mu<\lambda)$ is a $\Q$-basis for $\Ult_\calU(G)$.

\section{Forcing poset}

\begin{defin} \label{the.order}
We define $\calP$ as the set of the tuples of the form  $(E, \alpha, \calG, \xi, \phi)$ such that:

\begin{itemize}
    \item $E$ is a countable subset of $\lambda$ containing $\omega$,
    \item $\alpha < \frakc$,
    \item $\calG$ is a countable subset of $\calH$,
    \item $\xi=(\xi_{g}: g \in \calG$) is a family of elements of $\frakc\cap E$,
    \item $\phi:\, \Q^{(E)}\to \bbT^\alpha$ is a homomorphism,
    \item $\calU-\lim(\phi\circ g)=\phi(\chi_{\xi_g})$ for each $g\in\calG$,
    %\item $(\phi(\frac{n!}{S}\chi_n))_{n\in \omega}$ converges to $0 \in \bbT^{\alpha}$, for every positive integer $S\geq 1$.
    \end{itemize}
    
We define $(E, \alpha, \calG, \xi, \phi)\leq  (E', \alpha', \calG', \xi', \phi')$
if:

\begin{enumerate}
    \item $E\supseteq E'$,
    \item $\alpha \geq \alpha'$,
    \item $\calG \supseteq \calG'$,
    \item $\xi_{g}=\xi'_{g}$ for each $g \in \calG'$, and
    \item for every $\xi<\alpha'$ and $a \in \Q^{(E')}$, $\phi(a)(\xi)= \phi'(a)(\xi)$.
\end{enumerate}
 Given $p \in \calP$, we may denote its components by $E^p$, $\alpha^p$, $\mathcal G^p$, $\xi^p$ and $\phi^p$.

 If $H$ is a generic filter over $\calP$ then the generic homomorphism defined by $H$ is the mapping $\Phi$ of domain $\bigcup\{\dom(\phi^p): p \in H\}$ into $\bbT^\frakc$ defined by $\Phi(\cdot)(\xi)=\bigcup\{\phi^p(\cdot)(\xi):p \in H,\xi<\alpha^p\}$. In other words, if $p \in H$, $a \in \Q^{(E^p)}$ and $\xi<\alpha_p$, then $\Phi(a)(\xi)=\phi^p(a)(\xi)$.
\end{defin}

Naturally, we must assure that such generic homomorphisms are well-defined and into $\bbT^\frakc$. We will do so by showing that, assuming CH in the ground model, this forcing notion is $\omega_1$-closed and has the $\omega_2$-chain condition, and thus preserves cardinals and $\frakc$.

\begin{prop}\label{denseopen.Hausdorff}
Let $e\in G\bs\{0\}$. Then the set $\calC_e=\{p\in\calP:e\in\Q^{(E^p)}\text{ and }\phi^p(e)\ne0\}$ is open and dense in $\calP$.
\end{prop}
\begin{proof}
Openness: suppose $p\in\calC_e$ and $q\le p$. Then since $\phi^p(e)\ne0$, for some $\beta<\alpha^p$, $\phi^p(e)(\beta)\ne0$. By (2), $\beta<\alpha^p\le\alpha^q$, and by (1) $E^p\subseteq E^q$. It follows that $e\in\Q^{(E^q)}$ and, by (6), $\phi^q(e)(\beta)=\phi^p(e)(\beta)\ne0$. Thus, $q\in\calC_e$ as well.

Denseness: now let $p\in\calP$ be arbitrary. We shall produce a $q\le p$ such that $q\in\calC_e$. First, take any $d_0,d_1\in G\bs\{0\}$ such that $\supp e$, $\supp d_0$ and $\supp d_1$ are pariwise disjoint. Now take $C$ a countable subset of $\lambda$ such that $\omega\cup\supp e\cup\supp d_0\cup\supp d_1\cup\bigcup_{g\in\calG,n\in\omega}\supp g(n)\subseteq C$.

The Main Lemma of \cite{bellini&rodrigues&tomita2021.Qkappa} (Lemma 3.4) tells us that there exists a homomorphism $\rho:\Q^{(C)}\to\bbT$ such that:
\begin{itemize}
    \item $\rho(e)\ne0$ and $\rho(d_0)\ne\rho(d_1)$;
    \item $\calU-\lim(\rho\circ(\frac{1}{N}{g}))=\rho(\frac{1}{N}\chi_{\xi_g})$, for every $g\in\calG$ and $N\in\omega$.
\end{itemize}

Define now $E^q:=E^p\cup C$ and extend $\phi^p:\Q^{(E^p)}\to\bbT^{\alpha^p}$ to a $\varphi:\Q^{(E^q)}\to\bbT^{\alpha^p}$ using the divisibility of the codomain. Also extend $\rho$ to a $\psi:\Q^{(E^q)}\to\bbT$. Define then $\alpha^q=\alpha^p+1$, $\calG^q=\calG^p$, $\xi^q=\xi^p$, and $\phi^q=\varphi^\frown\psi$.

It follows that $q\in\calC_e$ and $q\le p$.
\end{proof}

\begin{prop}\label{denseopen.fullcodomain}
Let $\alpha<\frakc$. Then the set $\calA_\alpha=\{p\in\calP:\alpha^p>\alpha\}$ is open and dense in $\calP$.
\end{prop}
\begin{proof}
Openness: suppose $p\in\calA_\alpha$ and $q\le p$. We have that $\alpha^p>\alpha$ and $\alpha^q\ge\alpha^p$, and so $q\in\calA_\alpha$.

Denseness: let $p\in\calP$. If $\alpha^p>\alpha$, then $p\in\calA_\alpha$. Suppose now that $\alpha^p\le\alpha$. We define $q$ as follows: $E^q=E^p$, $\alpha^q=\alpha+1$, $\calG^q=\calG^p$, $\xi^q=\xi^p$, and $\phi^q={\phi^p}^\frown\psi$, where $\psi:\Q^{(E^q)}\to\bbT^{[\alpha^p,\alpha]}$ is the zero-homomorphism.

It follows that $q\in\calA_\alpha$ and $q\le p$.
\end{proof}

\begin{prop}\label{denseopen.allbasicsequences}
Let $g\in\calH$. Then the set $\calS_g:=\{p\in\calP:g\in\calG^p\}$ is open and dense in $\calP$.
\end{prop}
\begin{proof}
Openness: suppose $p\in\calS_g$ and $q\le p$. Since $g\in\calG^p$ and $\calG^p\subseteq\calG^q$, it follows that $g\in\calG^q$ and thus $q\in\calS_g$.

Denseness: Let $p\in\calP$ be given. First, take $E$ a countable subset of $\lambda$ such that $\omega\cup E^p\subseteq E$ and $\bigcup_{k\in\omega}\supp g(k)\subseteq E$. Take any $\mu\in\frakc\bs E$.

Define now: $E^q=E\cup\{\mu\}$, $\alpha^q=\alpha^p$, $\calG^q=\calG^p\cup\{g\}$, and $\xi^q=\xi^p\cup\{(g,\mu)\}$.

Extend $\phi^p:\Q^{(E^p)}\to\bbT^{\alpha^p}$ to a $\psi:\Q^{(E)}\to\bbT^{\alpha^p}$ using divisibility. Now, since $\bbT^{\alpha^p}$ is a compact space, let $z=\calU-\lim(\psi\circ g)$. Define then $\phi^q:\Q^{(E^q)}\to\bbT^{\alpha^q}$ as an extension of $\psi$, declaring that $\phi^q(\chi_\mu)=z$.

It follows that $q=(E^q,\alpha^q,\calG^q,\xi^q,\phi^q)\in\calS_g$ and $q\le p$.
\end{proof}

\begin{prop}\label{omega1closed}
$\calP$ is $\omega_1$-closed.
\end{prop}
\begin{proof}
Let $(p_t:t\in\omega)$ be a decreasing sequence in $\calP$. We shall produce an $r\in\calP$ such that $r\le p_t$ for all $t\in\omega$.

Denote now $p_t=(E^t,\alpha^t,\calG^t,\xi^t,\phi^t)$.

Define $E^r=\bigcup_{t\in\omega}E^t$, $\alpha^r=\sup_{t\in\omega}\alpha^t$, $\calG^r=\bigcup_{t\in\omega}\calG^t$ and given $g\in\calG^r$, $\xi^r_g=\xi^t_g$ for any $t$ such that $g\in\calG^t$ (this does not depend on such $t$'s).

Lastly, given $a\in\Q^{(E^r)}=\bigcup_{t\in\omega}\Q^{(E^t)}$ and $\xi<\alpha^r$, define $\phi^r(a)(\xi)=\phi^t(a)(\xi)$ for any $t$ such that $\xi<\alpha^t$ (again, this assignment does not depend on such $t$'s).
\end{proof}

We will now need a technical Lemma in order to guarantee the divergence of non-trivial sequences. This Lemma will require the use of the Main Lemma of \cite{bellini&rodrigues&tomita2021.Qkappa} and so for the sake of completeness we reproduce it here.

\begin{lem}[Main Lemma of \cite{bellini&rodrigues&tomita2021.Qkappa}] 
		Fix a selective ultrafilter $\calU$.
		Let $\mathcal F\subseteq G^\omega$ be a countable collection of distinct elements mod $\calU$ such that $([f]_\calU: f \in \mathcal F)\cup( [\chi_{\vec{\mu}}]_\calU:\, \mu \in \lambda)$ is $\mathbb Q$-linearly independent in $\text{Ult}_\calU(G)$.
		
		Let $d, d_0,d_1  \in G\setminus \{0\}$ with $\supp d$, $\supp d_0$, $\supp d_1$ pairwise disjoint, and $C$ be a countably infinite subset of $\lambda$ such that $\omega\cup \supp d\cup \supp d_0 \cup \supp d_1\cup \bigcup_{f \in \mathcal F, n \in \omega}\supp f(n) \subseteq C$. For each $f \in \mathcal F$, fix $\xi_f \in C$.
		
		Then there exists a homomorphism $\psi:\, {\mathbb Q}^{(C)}\longrightarrow {\mathbb T}$
		such that
		\begin{enumerate}[label=\alph*)]
			\item $\psi(d)\neq 0$, $\psi(d_0)\neq \psi(d_1)$, and
			\item $\calU$-$\lim (\psi (\frac{1}{N}f))=\psi(\frac{1}{N}\chi_{\xi_f})$, for each $f \in \mathcal F$ and $N\in \omega$.\qed
		\end{enumerate}
	\end{lem}

\begin{lem} \label{not.converges.lemma}

Let $\calG\subseteq\calH$ be countable and $B\in \calU$. Let $\calH'$ be a finite subset of $\calG$ and $(r_g:g\in\calH')$ a family of rational numbers. Let $E\subseteq\lambda$ countably infinite such that $\omega\bigcup\cup_{g\in\calG,n\in\omega}\supp g(n)\subseteq E$. Let $(\xi_g:g\in\calG)$ be a family in $\frakc\cap E$.
		
Then there exists a homomorphism $\phi:\, {\mathbb Q}^{(E)}\longrightarrow {\mathbb T}$
		such that
		\begin{enumerate}[label=\alph*)]
			
			\item $\calU$-$\lim (\phi (\frac{1}{N}g))=\phi(\frac{1}{N}\chi_{\xi_g})$, for each $g\in\calG$ and $N\in \omega$, and
			\item $(\phi(\sum_{g\in\calH'}r_g g(n)):n\in B)$ does not converge.	
		\end{enumerate}
	\end{lem}
	
	\begin{proof}  Let $B' \in \calU$ be a subset of $B$ such that  $(\sum_{g\in\calH'}r_g g(n):n\in B')$ is a one-to-one sequence, which is possible since the $g$'s are linearly independent mod $\calU$ with the constant sequences and by the selectiveness of $\calU$.
	
	Let ${\mathbb A}$ be an almost disjoint family on $B'$ of cardinality ${\mathfrak c}$ and $h_x:\, \omega \longrightarrow \{\sum_{g\in\calH'}r_g g(n):\, n\in x\}$ be a bijection for each $x \in \mathbb A$. \\

		\textbf{Claim}: There exist $x_0,x_1 \in {\mathbb A}$ such that $\{ [g]_\calU:\,g\in\calG\} \cup \{ [\chi_{\vec{\mu}}]_\calU:\, \mu \in \lambda \}\cup \{[h_{x_0}]_\calU,[h_{x_1}]_\calU\} $ is a linearly independent subset.\\
		
		\textbf{Proof of the claim:} Given $x_0, x_1 \in \mathbb A$, notice that $h_{x_0}(n)\neq h_{x_1}(n)$ for all but a finite numbers of $n$'s, so $[h_{x_0}]_\calU\neq [h_{x_1}]_\calU$. Since $\mathbb Q$ is countable, it follows that $\langle [h_{x}]_\calU: x \in \mathbb A\rangle$ has cardinality $\mathfrak c$, so there is $J\subseteq\mathbb A$ such that $|J|=\mathfrak c$ and that $([h_x]_\calU: x \in J)$ is linearly independent. Now notice that $\langle \calG\rangle \oplus \langle\chi_{\vec\xi}: \xi \in E\rangle$ is countable, so there exist $x_0, x_1 \in J$ such that $\{ [g]_\calU:\, g\in\calG\} \cup \{ [\chi_{\vec{\mu}}]_\calU:\, \mu \in E \}\cup \{[h_{x_0}]_\calU,[h_{x_1}]_\calU\} $ is linearly independent. Since all the supports of these elements are contained in $E$, it is straightforward to see that $\{ [g]_\calU:\, g\in\calG\} \cup \{ [\chi_{\vec{\mu}}]_\calU:\, \mu \in \lambda \}\cup \{[h_{x_0}]_\calU,[h_{x_1}]_\calU\} $ is  linearly independent. \qed\\
		
		We will now apply Lemma 3.4 (the Main Lemma) of \cite{bellini&rodrigues&tomita2021.Qkappa} with ${\mathcal F}=\{ g:\,g\in\calG\} \cup \{h_{x_0}, h_{x_1}\}$, $\xi _g=\xi_g$ for $g\in\calG$, $\xi_{h_{x_0}}= 0$, $\xi_{h_{x_1}}= 1$, $d_0=\chi_0$, $d_1=\chi_1$, $d=\chi_2$ and $C=E$. Take $\psi:\Q^{(E)}\to\bbT$ given by the conclusion of the Lemma.
		
		Clearly condition a) of this Lemma is satisfied. 
		
		Furthermore, 
		$(\psi(h_{x_i}(k)):k\in \omega)$ has $\psi(\chi_i)$ as an accumulation point for $i<2$. Since these sequences are reorderings of a subsequence of 
		$(\psi(\sum_{g\in\calH'}r_g g(n)):n\in B)$ and  $\phi(\chi_0)\neq \phi(\chi_1)$, it follows that b) is satisfied.
	\end{proof}

\begin{prop}\label{not.converges}
Let $h\in G^\omega$ be a one-to-one sequence. Then the set $\calE_h:=\{p\in\calP:\text{there is }\beta<\alpha^p\text{ such that }(\phi^p(h(n))(\beta):n\in\omega)\text{ does not converge}\}$
\end{prop}
\begin{proof}
Openness: Suppose $p\in\calE_h$ and $q\le p$. Then for some $\beta<\alpha^p$, $(\phi^p(h(n))(\beta):n\in\omega)$ does not converge. By (2), $\beta<\alpha^p\le\alpha^q$, and by (1) $E^p\subseteq E^q$. It follows that $h(n)\in\Q^{(E^q)}$ for all $n\in\omega$ and, by (6), $\phi^q(h(n))(\beta)=\phi^p(h(n))(\beta)$ for all $n\in\omega$. Since $(\phi^p(h(n))(\beta):n\in\omega)$ does not converge, then $(\phi^q(h(n))(\beta):n\in\omega)$ does not converge. Thus, $q\in\calE_h$ as well.

Denseness: Let $p\in\calP$ be given. Take $\calH'\subseteq\calH$ finite, $\mu_0,\ldots,\mu_{m-1}\in\lambda$ and $(r_g:g\in\calH')$ and $(s_{\mu_j}:j<m)$ families of rational numbers such that $[h]_\calU=\sum_{g\in\calH'}r_g [g]_\calU + \sum_{j<m}s_{\mu_j} [\chi_{\vec{\mu_j}}]_\calU$.

We will define a decreasing sequence of conditions $(p_i:i\in\omega)$. First, we obtain $p_0\le p$ by applying Proposition \ref{denseopen.allbasicsequences} $|\calH'|$ times in order to guarantee that $\calH'\subseteq\calG^{p_0}$.

Take an enumeration $(\gamma_k:k\ge1)$ of $\cup_{g\in\calH', n\in\omega}\supp g(n)\bigcup\cup_{n\in\omega}\supp h(n)\bigcup\{\mu_j:j<m\}$.

Apply Proposition \ref{denseopen.Hausdorff} using $e=\chi_{{\gamma_1}}$ and obtain $p_1\le p_0$ such that $\chi_{{\gamma_1}}\in\Q^{(E^{p_1})}$, which implies $\gamma_1\in E^{p_1}$.

Now recursively apply Proposition \ref{denseopen.Hausdorff} for each $i>1$ using $e=\chi_{{\gamma_i}}$ and obtain $p_i\le p_{i-1}$ such that $\chi_{{\gamma_i}}\in\Q^{(E^{p_i})}$, which implies $\gamma_i\in E^{p_i}$.

Applying Proposition \ref{omega1closed}, we obtain a $p_\omega\le p_i$ for all $i\in\omega$.

We will now use Lemma \ref{not.converges.lemma} with $\calG=\calG^{p_\omega}$, $\calH'=\calH'$, $r_g=r_g$ for $g\in\calH'$, $E=E^{p_\omega}$, $\xi_g=\xi^{p_\omega}_g$ for $g\in\calG^{p_\omega}$ and $B\in\calU$ such that $h(n)=\sum_{g\in\calH'}r_g g(n) + \sum_{j<m}s_{\mu_j} \chi_{{\mu_j}}$ for all $n\in B$. We thus obtain a $\psi:\Q^{(E^{p_\omega})}\to\bbT$ such that 
\begin{enumerate}[label=\alph*)]
			
			\item $\calU$-$\lim (\psi (\frac{1}{N}g))=\psi(\frac{1}{N}\chi_{\xi_g})$, for each $g\in\calG$ and $N\in \omega$, and
			\item $(\psi(\sum_{g\in\calH'}r_g g(n)):n\in B)$ does not converge.	
		\end{enumerate}
		
We define now a $q\le p_\omega$. Define $E^q=E^{p_\omega}$, $\alpha^q=\alpha^{p_\omega}+1$, $\calG^q=\calG^{p_\omega}$, $\xi^q=\xi^{p_\omega}$ and $\phi^q={\phi^{p_\omega}}^\frown\psi$.

Since $(\psi(\sum_{g\in\calH'}r_g g(n)):n\in B)$ does not converge and $\sum_{j<m}s_{\mu_j} \chi_{\mu_j}$ is constant, it follows that $(\psi(h(n)):n\in B)$ does not converge, and so $(\psi(h(n)):n\in\omega)$ does not converge.

Let $\beta=\alpha^p$. The definition $\phi^q={\phi^{p_\omega}}^\frown\psi$ means that for all $x\in\Q^{(E^q)}$, $\phi^q(x)(\beta)=\psi(x)$. Thus,  $(\psi(h(n)):n\in\omega)$ does not converge means that $(\phi^q(h(n))(\beta):n\in\omega)$ does not converge, proving that $q\in\calE_h$.

\end{proof}

\begin{prop}\label{omega2cc}
Assume CH. Then the partial order $\calP$ has the $\omega_2$-chain condition.
\end{prop}
\begin{proof}
Since under CH, $\frakc^+=\omega_2$, we will show that $\calP$ has the $\frakc^+$-c.c.. So let $\calQ\subseteq\calP$ of cardinality $\frakc^+$. We will show that $\calQ$ has a subset consisting of $\frakc^+$ pairwise compatible elements.

First, take a $\calQ_0\subseteq\calQ$ of cardinality $\frakc^+$ and an $\alpha<\frakc$ such that $\alpha_q=\alpha$ for each $q\in\calQ_0$.

Using the $\Delta$-system Lemma, take a $\calQ_1\subseteq\calQ_0$ of cardinality $\frakc^+$ such that $\{E^q:q\in\calQ_1\}$ is a $\Delta$-system of root $\tilde E$. Now, using CH, we have that $\left|{(\bbT^\alpha)}^{\Q^{(\tilde E)}}\right|=\frakc$, and thus we may take a $\calQ_2\subseteq\calQ_1$ of cardinality $\frakc^+$ such that ${\phi^q}|_{\Q^{(\tilde E)}}={\phi^p}|_{\Q^{(\tilde E)}}$ for all $q,p\in\calQ_2$.

Using the $\Delta$-system Lemma again, take $\calQ_3\subseteq\calQ_2$ of cardinality $\frakc^+$ such that $\{\calG^q:q\in\calQ_3\}$ is a $\Delta$-system of root $\tilde\calG$. Now since for each $q\in\calQ_3$, $(\xi^q_g:g\in\tilde\calG)\in\frakc^{\tilde\calG}$ and $|\frakc^{\tilde\calG}|=\frakc$, take $\calQ_4\subseteq\calQ_3$ of cardinality $\frakc^+$ such that for all $q,p\in\calQ_4$ and each $g\in\tilde\calG$, $\xi^q_g=\xi^p_g$.

A common extension to $q,p\in\calQ_4$ is an $r$ defined as follows: $E^r=E^q\cup E^p$; $\alpha^r=\alpha^q=\alpha^p$; $\calG^r=\calG^q\cup\calG^p$; $\xi^r=\xi^q\cup\xi^p$.

To define $\phi^r$, notice that $\Q^{(E^r)}=\Q^{(E^q\bs\tilde E)}\oplus\Q^{(\tilde E)}\oplus\Q^{(E^p\bs\tilde E)}$. Then, let $\pi_0:\Q^{(E^r)}\to\Q^{(E^q\bs\tilde E)}$, $\pi_1:\Q^{(E^r)}\to\Q^{(\tilde E)}$ and $\pi_2:\Q^{(E^r)}\to\Q^{(E^p\bs\tilde E)}$ be the projections. Define $\phi^r=\phi^q\circ\pi_0+\phi^q\circ\pi_1+\phi^p\circ\pi_2=\phi^q\circ\pi_0+\phi^p\circ\pi_1+\phi^p\circ\pi_2$ and we are done.
\end{proof}

\begin{teo}\label{generic.monomorphism.theorem}
Assume CH. Then the forcing notion $\calP$ preserves cofinalities and cardinals, preserves $\frakc$ and does not add reals. Given $H$ a $\calP$-generic filter, the associated $\Phi$ (as in \ref{the.order}) is a well-defined monomorphism from $G$ to $\bbT^\frakc$. Also, given any $g\in\calH$, there exists a $\xi\in\frakc$ such that $\calU-\lim(\Phi\circ g)=\Phi(\chi_\xi)$. Furthermore, given any one-to-one sequence $h\in G^\omega$, $\Phi\circ h$ does not converge.
\end{teo}
\begin{proof}
By Propositions \ref{omega1closed} and \ref{omega2cc}, $\calP$ preserves cofinalities (and therefore cardinals), does not add reals and preserves $\frakc$. Note that since being a basis for $\Ult_\calU(G)$ is absolute for transitive models of ZFC, $\calH$ is still a basis for $\Ult_\calU(G)$ in the extension.

Let $H$ be a $\calP$-generic filter and let $\Phi$ be its associated homomorphism.

First, let us see that $\Phi:G\to\bbT^\frakc$ is well-defined. Let $q,p\in H$ and suppose $\xi<\min\{\alpha^q,\alpha^p\}$ and $e\in\Q^{(E^q\cap E^q)}$. We must see that $\phi^q(e)(\xi)=\phi^p(e)(\xi)$. Take $r\in H$ such that $r\le q,p$. Then $\xi<\alpha^r$ and $e\in\Q^{(E^r)}$, and by item (5) of Definition \ref{the.order}, $\phi^q(e)(\xi)=\phi^r(e)(\xi)=\phi^p(e)(\xi)$.

Now let $\alpha<\frakc$ and $e\in\Q^{(\lambda)}$ such that $e\ne0$. Since $\calC_e$ and $\calA_\alpha$ are open and dense, let $p\in H$ such that $e\in\Q^{(E^p)}$, $\phi^p(e)\ne0$ and $\alpha^p>\alpha$. Since $\phi^p(e)\ne0$, there is a $\xi\le\alpha^p$ such that $\phi^p(e)(\xi)\ne0$, and therefore $\Phi(e)(\xi)\ne0$, so that $\Phi(e)\ne0$. And since $\alpha\subseteq\alpha^p\subseteq\dom\Phi(e)\subseteq\frakc$, and $\alpha$ was arbitrary, it follows that $\dom\Phi(e)=\frakc$.

We have thus seen that the domain of $\Phi$ is $\Q^{(\lambda)}$, the codomain is $\bbT^\frakc$ and that $\Phi$ is injective.

Now we see that $\Phi$ is a homomorphism. Let $e,e'\in\Q^{(\lambda)}$. Since $\calC_e$, $\calC_{e'}$ and $\calC_{e+e'}$ are dense and open, take $p\in H$ such that $e,e',e+e'\in\Q^{(E^p)}$. We know that $\phi^p$ is a homomorphism, and so $\Phi(e+e')=\phi^p(e+e')=\phi^p(e)+\phi^p(e')=\Phi(e)+\Phi(e')$.

Let now $g\in\calH$. Since $\calS_g$ is open and dense, let $p\in H$ such that $g\in\calG^p$. We have then that $\calU-\lim(\phi^p\circ g)=\phi^p(\xi_g)$. Let us see that $\calU-\lim(\Phi\circ g)=\Phi(\xi_g)$. Let $F$ be a finite subset of $\frakc$ and let $\alpha<\frakc$ such that $F\subseteq\alpha$. Since $\calA_\alpha$ is open and dense, and $p\in H$, let $q\in H$ such that $q\le p$ and $\alpha^q>\alpha$. We have then that $\calU-\lim(\pi_F\circ\phi^q\circ g)=(\pi_F\circ\phi^q)(\xi_g)$. Since $\pi_F\circ\phi^q=\pi_F\circ\Phi$ (due to $\alpha^q\supseteq F$), it follows that $\calU-\lim(\pi_F\circ\Phi\circ g)=(\pi_F\circ\Phi)(\xi_g)$, as we sought for.

Finally, let $h\in G^\omega$ be a one-to-one sequence. Since $\calE_h$ is open and dense, let $p\in H\cap\calE_h$. Take then a $\beta<\alpha^p$ such that $(\phi^p(h(n))(\beta):n\in\omega)$ does not converge. Since $\Phi(h(n))(\beta)=\phi^p(h(n))(\beta)$ for each $n\in\omega$, it follows that  $(\Phi(h(n))(\beta):n\in\omega)$ does not converge, which in turn implies that  $(\Phi(h(n)):n\in\omega)=\Phi\circ h$ does not converge.
\end{proof}

\begin{teo}
It is consistent with ZFC that given $\lambda$ a countably cofinal cardinal and $\calU$ a selective ultrafilter, $G$ can be endowed with a $\calU$-compact Hausdorff group topology without non-trivial convergent sequences.
\end{teo}
\begin{proof}
Consider the forcing model obtained via forcing with $\calP$. We fix a generic monomorphism $\Phi$ as in \ref{generic.monomorphism.theorem}. Since we have that $\Phi$ is a monomorphism from $G$ to $\bbT^\frakc$, then $\Phi$ induces a Hausdorff group topology in $G$ such that for any $g\in\calH$, there exists a $\xi\in\frakc$ such that $\calU-\lim g=\chi_\xi$. Fix one such $\xi_g$ for each $g\in\calH$.

Now let us see that such topology is indeed $\calU$-compact. Let $f\in G^\omega$. Since $([g]_\calU:g\in\calH)\cup([\chi_{\vec\mu}]_\calU:\mu<\lambda)$ is a $\Q$-basis for $\Ult_\calU(G)$, there exist families $(r_g:g\in\calH)$ and $(s_\mu:\mu<\lambda)$ of rational numbers, all but finitely many of which are 0, such that $[f]_\calU=\sum_{g\in\calH}r_g\cdot[g]_\calU+\sum_{\mu<\lambda}s_\mu\cdot[\chi_{\vec\mu}]_\calU$. It follows then that $\calU-\lim f=\sum_{g\in\calH}r_g\cdot(\calU-\lim g)+\sum_{\mu<\lambda}s_\mu\cdot(\calU-\lim\chi_{\vec\mu})=\sum_{g\in\calH}r_g\cdot\chi_{\xi_g}+\sum_{\mu<\lambda}s_\mu\cdot\chi_\mu$ and $f$ has a $\calU$-limit.

Finally, there are no non-trivial convergent sequences since for each $h\in G^\omega$ one-to-one, $\Phi\circ h$ does not converge, which means that in the induced topology on $G$, $h$ does not converge.
\end{proof}

\section{Questions}

The recent breakthrough by Hru\v s\' ak, Ramos, Shelah and van Mill (\cite{hrusakvanmillgarciashelah2020}) obtained a countably compact group without non-trivial convergent sequences in ZFC. Their construction is over a $\frakc$-sized group of order 2. Shortly after Tomita and Trianon-Fraga (\cite{tomita&trianonfraga2022}) expanded this construction to obtain a countably compact group without non-trivial convergent sequences which has cardinality $2^\frakc$ and order 2. It is noteworthy that these constructions did not yield so far a $\calU$-compact group for any ultrafilter $\calU$ and do not seem to be suitable for non-torsion groups. (We recall that such results also imply the existence of a countably compact group whose square is not countably compact; therefore, not all countably compact groups are $\calU$-compact for some $\calU$).

As such, the following questions remain open:

\begin{question}
Is there a countably compact non-torsion group without non-trivial convergent sequences in ZFC?
\end{question}

\begin{question}
How large can countably compact groups without non-trivial convergent sequences be in ZFC?
\end{question}

\begin{question}
Is there a $\calU$-compact group without non-trivial convergent sequences in ZFC, for some ultrafilter $\calU$?
\end{question}

It should be noted that a positive answer to the first question implies the negation of Wallace's conjecture, that is, it implies the existence of a countably compact semigroup with both-sided cancellation which is not a group.

\end{document}